\documentclass[12pt, twoside, leqno]{article}
\usepackage{amsfonts}
\usepackage{amssymb}
\usepackage{amsthm}
\usepackage{amsmath, amsthm}
\usepackage{enumitem}

\newtheorem{thm}{Theorem}[section]
\newtheorem{lem}[thm]{Lemma}
\newtheorem{prop}[thm]{Proposition}

\newtheorem{defn}{Definition}

\begin{document}
\baselineskip=17pt

\title{Complexity of Scott Sentences}

\author{Rachael Alvir\\
Department of Mathematics\\
University of Notre Dame\\
255 Hurley\\
Notre Dame, IN 46556, USA\\
E-mail: ralvir@nd.edu
\and
Julia F.\ Knight\\
Department of Mathematics\\
University of Notre Dame\\
255 Hurley\\
Notre Dame, IN 46556, USA\\
E-mail: Julia.F.Knight.1@nd.edu\\
\and
Charles McCoy CSC\\
Department of Mathematics\\
University of Portland\\
5000 N. Willamette Blvd.,\\
Portland, Oregon 97203, USA\\
E-mail: mccoy@up.edu}

\date{\today}

\maketitle

\renewcommand{\thefootnote}{}



\renewcommand{\thefootnote}{\arabic{footnote}}
\setcounter{footnote}{0}

\begin{abstract}

We give effective versions of some results on Scott sentences.  We show that if $\mathcal{A}$ has a computable $\Pi_\alpha$ Scott sentence, then the orbits of all tuples are defined by formulas that are computable $\Sigma_\beta$ for some $\beta <\alpha$.  (This is an effective version of a result of Montalb\'{a}n \cite{Montalban}.)  We show that if a countable structure $\mathcal{A}$ has a computable $\Sigma_\alpha$ Scott sentence and one that is computable $\Pi_\alpha$, then it has one that is computable $d$-$\Sigma_\beta$ for some $\beta < \alpha$.  (This is an effective version of a result of A.\ Miller \cite{AMiller}.)  We also give an effective version of a result of D.\ Miller \cite{DMiller}.  Using the non-effective results of Montalb\'{a}n and A.\ Miller, we show that a finitely generated group has a $d$-$\Sigma_2$ Scott sentence iff the orbit of some (or every) generating tuple is defined by a $\Pi_1$ formula.  Using our effective results, we show that for a computable finitely generated group, there is a computable $d$-$\Sigma_2$ Scott sentence iff the orbit of some (every) generating tuple is defined by a computable $\Pi_1$ formula.          

\end{abstract}

\section{Introduction}

The $L_{\omega_1\omega}$-formulas are infinitary formulas in which the disjunctions and conjunctions are over countable sets, and the strings of quantifiers are finite.  We consider $L_{\omega_1\omega}$-formulas with only finitely many free variables.  There is no prenex normal form for $L_{\omega_1\omega}$-formulas.  We cannot, in general, bring the quantifiers to the front.  However, we can bring the negations inside, and this gives a kind of normal form.  Formulas in this normal form are classified as $\Sigma_\alpha$ or $\Pi_\alpha$ for countable ordinals $\alpha$.  

\begin{enumerate}

\item  A formula $\varphi(\bar{x})$ is $\Sigma_0$ and $\Pi_0$ if it is finitary quantifier-free.

\item  Let $\alpha > 0$.  

\begin{enumerate}

\item  $\varphi(\bar{x})$ is $\Sigma_\alpha$ if it is a countable disjunction of formulas $(\exists\bar{u})\psi(\bar{x},\bar{u})$, where $\psi$ is $\Pi_\beta$ for some $\beta < \alpha$.

\item  $\varphi(\bar{x})$ is $\Pi_\alpha$ if it is a countable conjunction of formulas $(\forall\bar{u})\psi(\bar{x},\bar{u})$, where $\psi$ is $\Sigma_\beta$ for some $\beta < \alpha$. 

\end{enumerate}

\end{enumerate}

We use special notation for some further classes of formulas.

\begin{enumerate}

\item  A formula is $\Sigma_{<\alpha}$ (resp. $\Pi_{<\alpha}$) if it is $\Sigma_\beta$ (resp. $\Pi_\beta$) for some $\beta <\alpha$.

\item  A formula is $d$-$\Sigma_\alpha$ if it is the conjunction of a formula that is $\Sigma_\alpha$ and one that is $\Pi_\alpha$.

\end{enumerate}

\noindent
\textbf{Negations}.  For a $\Sigma_\alpha$ (resp. $\Pi_\alpha$) formula $\varphi$, in normal form, we write $neg(\varphi)$ for the natural $\Pi_\alpha$ (resp. $\Sigma_\alpha$) formula in normal form that is logically equivalent to the negation of $\varphi$.   

\bigskip
\noindent
\textbf{Computable infinitary formulas}.  Roughly speaking, the \emph{computable infinitary formulas} are formulas of $L_{\omega_1\omega}$ in which the infinite disjunctions and conjunctions are over c.e.\ sets.  For more about computable infinitary formulas, see \cite{AK}.  We classify the computable infinitary formulas as computable $\Sigma_\alpha$ or computable $\Pi_\alpha$ for computable ordinals $\alpha$. We may refer to computable $\Sigma_{<\alpha}$ formulas, or computable $d$-$\Sigma_\alpha$ formulas.

\bigskip

Scott \cite{Scott} proved the following.

\begin{thm} [Scott Isomorphism Theorem]

Let $\mathcal{A}$ be a countable structure for a countable language $L$.  Then there is a sentence of $L_{\omega_1\omega}$ whose countable models are just the isomorphic copies of $\mathcal{A}$.  

\end{thm}

A sentence with the property above is called a \emph{Scott sentence} for $\mathcal{A}$.  The complexity of an optimal Scott sentence for a structure $\mathcal{A}$ measures the internal complexity of $\mathcal{A}$.   

\bigskip

For a countable language $L$, let $C$ be a countably infinite set of new constants.  Identifying the constants with natural numbers, we suppose that $C = \omega$.  Let $Mod(L)$ be the class of $L$-structures with universe $\omega$, and for a computable infinitary sentence $\psi$ in the language $L \cup C$, let $Mod(\psi)$ be the class of $L$-structures with universe $\omega$ that satisfy $\psi$.  There is a natural topology on $Mod(L)$, generated by basic open (actually clopen) neighborhoods of the form $Mod(\varphi)$, for $\varphi$ a finitary quantifier-free sentence in the language $L\cup C$.  We define the Borel hierarchy of classes $K\subseteq Mod(L)$ as follows.

\begin{enumerate}

\item  $K$ is $\mathbf{\Sigma_0}$ and $\mathbf{\Pi_0}$ if it is a basic clopen neighborhood.

\item  For $0 < \alpha < \omega_1$, 

\begin{enumerate}

\item  $K$ is $\mathbf{\Sigma_\alpha}$ if it is a countable union of sets each of which is $\mathbf{\Pi_\beta}$ for some $\beta < \alpha$.

\item  $K$ is $\mathbf{\Pi_\alpha}$ if it is a countable intersection of sets each of which is $\mathbf{\Sigma_\beta}$ for some $\beta<\alpha$.

\item $K$ is $\mathbf{d}$-$\mathbf{\Sigma_{\alpha}}$ if it is a difference of $\Sigma_{\alpha}$ sets.

\end{enumerate}

\end{enumerate}

Vaught \cite{Vaught} proved the following.

\begin{thm} [Vaught]
\label{Vaught}

For a set $K\subseteq Mod(L)$, closed under isomorphism, $K$ is $\mathbf{\Pi_\alpha}$ in the Borel hierarchy iff there is a 
$\Pi_\alpha$ sentence $\varphi$ of $L_{\omega_1\omega}$ such that $K = Mod(\varphi)$.  

\end{thm}
It is easy to see, as a corollary of Vaught's Theorem, that the same holds for $\Sigma_{\alpha}$ and $d$-$\Sigma_{\alpha}$ sets and sentences.

If $L$ is a computable language, then we have also the \emph{effective} Borel hierarchy.  Let $K\subseteq Mod(L)$.  

\begin{enumerate}

\item  $K$ is effective $\Sigma_0$ and effective $\Pi_0$ if it is a basic clopen neighborhood.

\item  For $0 < \alpha < \omega_1^{CK}$, 

\begin{enumerate}

\item  $K$ is effective $\Sigma_\alpha$ if it is a c.e.\ union of sets each of which is effective $\Pi_\beta$ for some $\beta < \alpha$,

\item  $K$ is effective $\Pi_\alpha$ if it is a c.e.\ intersection of sets each of which is effective $\Sigma_\beta$ for some $\beta < \alpha$.

\end{enumerate}

\end{enumerate}

We may effectively identify elements of $Mod(L)$ with elements of $2^\omega$, and then the effective Borel sets are exactly the hyperarithmetical sets of functions in $2^\omega$.  Vanden Boom \cite{VB} proved the effective analogue of Vaught's Theorem.   

\begin{thm} [Vanden Boom]
\label{Vanden Boom}

For a set $K\subseteq Mod(L)$, closed under isomorphism, $K$ is $\Pi_\alpha$ in the effective Borel hierarchy iff there is a computable $\Pi_\alpha$ sentence $\varphi$ of $L_{\omega_1\omega}$ such that $K = Mod(\varphi)$.  

\end{thm}

Montalb\'{a}n \cite{Montalban} proved that for a countable ordinal $\alpha\geq 1$, a countable structure $\mathcal{A}$ has a $\Pi_{\alpha+1}$ Scott sentence iff the orbits of all tuples are defined by $\Sigma_\alpha$ formulas.  The implication $\Leftarrow$ is as in the proof of Scott's Theorem.  For the implication $\Rightarrow$, Montalb\'{a}n's proof was clever.  We shall use the ideas from his proof to obtain further results.      

\bigskip

In Section 2, we show that for a countable ordinal $\alpha\geq 2$, if $\mathcal{A}$ has a $\Pi_\alpha$ Scott sentence, then the orbits of all tuples in $\mathcal{A}$ are defined by $\Sigma_{<\alpha}$ formulas.  (Montalb\'{a}n's Theorem gives this in the case where $\alpha$ is a successor ordinal.)
For limit $\alpha$, the implication $\Leftarrow$ fails.  There are familiar structures $\mathcal{A}$ such that the orbits of all tuples are defined by $\Sigma_{<\omega}$ formulas but there is no $\Pi_\omega$ Scott sentence.  
Next, we give an effective version of Montalb\'{a}n's theorem, saying that for a computable ordinal $\alpha\geq 2$, if $\mathcal{A}$ has a computable $\Pi_\alpha$ Scott sentence, then the orbits of all tuples are defined by computable $\Sigma_{<\alpha}$ formulas.  Even for successor ordinals $\alpha$, the implication $\Leftarrow$ fails.    
We construct an example of a computable structure such that the orbits of all tuples are defined by finitary quantifier-free formulas, but there is no computable $\Pi_2$ Scott sentence.
    
In Section 3, we consider further results that can be proved using ideas from Section 2.  A.\ Miller showed that if $\mathcal{A}$ has a $\Pi_\alpha$ Scott sentence and a $\Sigma_\alpha$ Scott sentence, then it has a Scott sentence that is $d$-$\Sigma_{<\alpha}$.  The proof was based on a result of D.\ Miller \cite{DMiller} on separators for disjoint sets axiomatized by $\Pi_\alpha$ sentences.  
Our Theorem \ref{effectiveAMiller} is an effective version of the result of A.~Miller, saying that if $\mathcal{A}$ has a computable $\Pi_\alpha$ Scott sentence and a computable $\Sigma_\alpha$ Scott sentence, then it has a Scott sentence that is computable $d$-$\Sigma_{<\alpha}$. In \cite{DMiller}, D.\ Miller gave an effective version of his result, which, unfortunately, was not sufficient to prove Theorem \ref{effectiveAMiller}.  We give a direct proof of Theorem \ref{effectiveAMiller}.  Then, in Theorem \ref{effectiveDMiller}, we give an effective version of the result of D.\ Miller that would have served to prove Theorem \ref{effectiveAMiller}.

In Section 4, we consider finitely generated groups.  For such a group, there is always a $\Sigma_3$ Scott sentence, and if the group is computable, then there is a computable $\Sigma_3$ Scott sentence (see \cite{KS}).  Often, however, there is a simpler Scott sentence.  In fact, the second author had conjectured that every finitely generated group has a $d$-$\Sigma_2$ Scott sentence, and every computable finitely generated group has a computable $d$-$\Sigma_2$ Scott sentence.  Recently, Harrison-Trainor and Ho \cite{HH} gave an example disproving both conjectures.  We show that if $G$ is a finitely generated group, then $G$ has a $d$-$\Sigma_2$ Scott sentence iff there is a generating tuple whose orbit is defined by a $\Pi_1$ formula, and if $G$ is a computable finitely generated group, then $G$ has a computable $d$-$\Sigma_2$ Scott sentence iff there is a generating tuple whose orbit is defined by a computable $\Pi_1$-formula. 

Recall that the definition above of $Mod(L)$, all $L$-structures have universe $\omega$.  Throughout this paper, all structures given are assumed to be countably infinite with universe $\omega$, and all structures we build are guaranteed by our techniques to have universe $\omega$.

\section{Varying Montalb\'{a}n's Theorem}

Montalb\'{a}n \cite{Montalban} proved the following.

\begin{thm} [Montalb\'{a}n]
\label{Montalban}

Suppose $\alpha\geq 1$ is a countable ordinal, and let $\mathcal{A}$ be a countable structure for a countable language $L$.  Then $\mathcal{A}$ has a $\Pi_{\alpha+1}$ Scott sentence iff the automorphism orbit of each tuple is defined by a $\Sigma_\alpha$ formula.

\end{thm}

For the implication $\Leftarrow$, the proof is as for Scott's Isomorphism Theorem.  

\begin{proof} [Proof of $\Leftarrow$]

For each $\bar{a}$, let $\varphi_{\bar{a}}(\bar{x})$ be a $\Sigma_{\alpha}$ formula that defines the orbit of $\bar{a}$.  For each $\bar{a}$, we determine a sentence $\rho_{\bar{a}}$ as follows.   

\begin{itemize}

\item  $\rho_\emptyset$: $\bigwedge_b (\exists x)\varphi_b(x)\ \&\ (\forall x)\bigvee_b\varphi_b(x)$

\item  $\rho_{\bar{a}}$:  $(\forall\bar{u})[\varphi_{\bar{a}}(\bar{u})\rightarrow
(\bigwedge_b (\exists x)\varphi_{\bar{a},b}(\bar{u},x)\ \&\ (\forall x)\bigvee_b\varphi_{\bar{a},b}(x))]$
\end{itemize}
Our Scott sentence is the conjunction of the sentences $\rho_{\bar{a}}$.  It is not difficult to see that this is $\Pi_{\alpha+1}$.
\end{proof}

The implication $\Rightarrow$ in Montalb\'{a}n's result also holds for limit ordinals, with no change in the proof.  Here is the statement.      

\begin{thm}
\label{MontalbanVariant}

Let $\mathcal{A}$ be a countable structure for a countable language $L$.  Let $\alpha\geq 2$.  If $\mathcal{A}$ has a $\Pi_\alpha$ Scott sentence, then the orbit of each tuple is defined by a $\Sigma_{<\alpha}$ formula.

\end{thm} 

In our account of Montalb\'{a}n's proof, we use a ``consistency property.''  This is a family of sets of sentences arising in a kind of Henkin construction, developed by Makkai for producing models of $L_{\omega_1\omega}$ sentences.  See \cite{Keisler} for a discussion of consistency properties.  The definition that we give below is not standard, but it suits our needs.      

\begin{defn}

Let $L$ be a countable language, and let $C$ be a countably infinite set of new constants.  A \emph{consistency property} is a non-empty set $\mathcal{C}$ of finite sets $S$ of sentences, each obtained by substituting constants from $C$ for the free variables in an $L_{\omega_1\omega}$ formula in normal form, such that the following conditions hold:

\begin{enumerate}

\item  for $S\in\mathcal{C}$, if $\varphi\in S$, where $\varphi = \bigwedge_i (\forall\bar{u}_i)\varphi_i(\bar{u}_i)$, then for each $i$ and each appropriate tuple of constants $\bar{c}$, there exists $S'\supseteq S$ in $\mathcal{C}$ with $\varphi_i(\bar{c})\in S'$,

\item  for $S\in\mathcal{C}$,
if $\varphi\in S$, where $\varphi = \bigvee_i (\exists\bar{u}_i)\varphi_i(\bar{u}_i)$, then for some $i$ and $\bar{c}$, there exists $S'\supseteq S$ in $\mathcal{C}$ with $\varphi_i(\bar{c})\in S'$, 

\item  for $S\in\mathcal{C}$, for each finitary quantifier-free $L$-formula $\varphi(\bar{x})$ and appropriate tuple $\bar{c}$, there exists $S'\supseteq S$ in $\mathcal{C}$ with $\pm\varphi(\bar{c})\in S'$,

\item  for $S\in\mathcal{C}$, if $F$ is an $n$-place function symbol, and $c_1,\ldots,c_n\in C$, there is a constant $d\in C$ such that for some $S'\supseteq S$ in $\mathcal{C}$, the sentence $F(c_1,\ldots,c_n) = d$ is in $S'$, 

\item  for $S\in\mathcal{C}$, and distinct $c,c'\in C$, the sentence $c = c'$ is not in $S$, 

\item  for $S\in\mathcal{C}$, the set of finitary quantifier-free sentences in $S$ is consistent.

\end{enumerate}

\end{defn}

\begin{lem}

If $\mathcal{C}$ is a consistency property, then we can form a countable chain $(S_n)_{n\in\omega}$ of elements of $\mathcal{C}$ such that the set $\{S_n:n\in\omega\}$ is also a consistency property.  

\end{lem}

\begin{proof} [Proof sketch]

This is clear just from the fact that the sets $S$ in $\mathcal{C}$ are finite, and there are only finitely many clauses (in the definition of consistency property) asking for extensions.
\end{proof}

\begin{prop}

Let $(S_n)_{n\in\omega}$ be a countable chain such that $\{S_n:n\in\omega\}$ is a consistency property.  
Then the set of atomic sentences and negations of atomic sentences in $\cup_n S_n$ is the atomic diagram of a well defined structure $\mathcal{B}$, with universe equal to $C$. Moreover, all sentences $\varphi\in \cup_n S_n$ are true of the appropriate elements in $\mathcal{B}$.

\end{prop}

\begin{proof}

For an $n$-placed relation symbol $R$, we let $R^\mathcal{B}$ be the set of $(c_1,\ldots,c_n)$ such that $R(c_1,\ldots,c_n)\in \cup_n S_n$.  Conditions (3) and (6) guarantee that $R^\mathcal{B}$ is a well-defined relation on $C^n$.  Suppose $F$ is an $n$-placed function symbol in $L$. Then $F^\mathcal{B}(c_1,\ldots,c_n) = d$ if the sentence $F(c_1,\ldots,c_n) = d$ is in some $S_n$.  Conditions (4), (5), and (6) guarantee that $F^\mathcal{B}$ is well-defined. For $c,d\in C$, the sentence $c = d$ is in $\cup_n S_n$ iff the constants $c$ and $d$ are actually the same.  An easy induction on terms $\tau(\bar{x})$ shows that $\tau^\mathcal{B}(\bar{c}) = d$ iff the sentence $\tau(\bar{c}) = d$ is in $\cup_n S_n$.  Then an easy induction on finitary quantifier-free formulas $\varphi$ shows that $\mathcal{B}\models\varphi(\bar{c})$ iff the sentence $\varphi(\bar{c})$ is in $\cup_n S_n$.
Finally, an easy induction on sentences shows that if $\varphi\in\cup_n S_n$, then $\mathcal{B}\models\varphi$.    
\end{proof}       

We want a consistency property $\mathcal{C}$ that produces models of the $\Pi_\alpha$ Scott sentence 
$\varphi = \bigwedge_i(\forall\bar{u}_i)\varphi_i(\bar{u}_i)$, where $\varphi_i(\bar{u}_i) = \bigvee_j (\exists\bar{v}_{i,j})\psi_{i,j}(\bar{u}_i,\bar{v}_{i,j})$.  We also want to control the complexity of the sentences that appear in $S\in\mathcal{C}$.  Instead of putting $\varphi$ into various sets $S\in \mathcal{C}$, we add to the six conditions in the definition of consistency property a seventh condition guaranteeing that $\varphi$ is witnessed.     
\begin{enumerate}
\setcounter{enumi}{6}

\item  for $S\in\mathcal{C}$, for each $i$ and each appropriate $\bar{c}$, there exist $j$, and an appropriate $\bar{d}$ such that for some $S'\supseteq S$ in $\mathcal{C}$, $\psi_{i,j}(\bar{c},\bar{d})\in S'$.  

\end{enumerate}

If $\mathcal{C}$ is a consistency property satisfying Conditions (1)--(7), then there is a chain $(S_n)_{n\in\omega}$ of elements of $\mathcal{C}$ such that $\{S_n:n\in\omega\}$ also satisfies Conditions (1)--(7).  For any such chain $(S_n)_{n\in\omega}$ the resulting structure is a model of $\varphi$.  Our consistency property $\mathcal{C}$ will be the set of finite sets $S$ of sentences in the language $L\cup C$, each $\Sigma_\beta$ or $\Pi_\beta$ for some $\beta$ such that $\beta+1 <\alpha$ (and each, recall, obtained by substituting constants from $C$ for the free variables in an $L_{\omega_1\omega}$ formula in normal form), where some interpretation of the constants from $C$ appearing in the sentences of $S$, mapping distinct constants to distinct elements of $\mathcal{A}$, makes all of these sentences true.  (Since the set of sentences is finite, there are only finitely many constants from $C$ to assign to elements of $\mathcal{A}$.)  It is easy to see that this satisfies Conditions (1)--(7).  

\bigskip
\noindent
\textbf{Claim}:  For each tuple $\bar{a}$ of distinct elements, our $\mathcal{C}$ fails to satisfy the following further condition.       

\begin{enumerate}
\setcounter{enumi}{7}

\item  for each $S\in\mathcal{C}$, and for each $\bar{c}$, a tuple of distinct constants of the same length as $\bar{a}$, there is some $\Pi_{<\alpha}$ formula $\psi(\bar{x}) = \bigwedge_i (\forall\bar{u}_i)\psi_i(\bar{x},\bar{u}_i)$, true of $\bar{a}$, and some $S'\supseteq S$ in $\mathcal{C}$ such that for some $i$ and some $\bar{d}$, $neg(\psi_i(\bar{c},\bar{d}))\in S'$.  (Note that for some $\beta$ such that $\beta+1 <\alpha$, $\psi_i(\bar{c},\bar{d})$ is $\Sigma_\beta$, and $neg(\psi_i(\bar{c},\bar{d}))$ is $\Pi_\beta$.)  

\end{enumerate}

\begin{proof} [Proof of Claim]

If our $\mathcal{C}$ satisfied Conditions (1)--(8),  then, we would have countable chains $(S_n)_{n\in\omega}$ yielding models of $\varphi$ with no tuple satisfying all of the $\Pi_{<\alpha}$ formulas true of $\bar{a}$.  Since $\varphi$ is a Scott sentence for $\mathcal{A}$, this is impossible.  
\end{proof}

By the Claim, there must be some set of sentences $S\in\mathcal{C}$ and some tuple of distinct constants $\bar{c}$ from $C$, of the same length as $\bar{a}$, that witness the failure of Condition (8).  Let $\bar{c}'$ be the tuple of all constants from $C$, other than $\bar{c}$, appearing in $S$, and let $\chi(\bar{c},\bar{c}')$ be the finite conjunction of the sentences in $S$ and sentences expresssing that the elements of $\bar{c}, \bar{c}'$ are pairwise distinct.  Then $\mathcal{A}$ satisfies the $\Sigma_{<\alpha}$ sentence $(\exists\bar{x})\chi(\bar{a},\bar{x})$; and for all $\Pi_{<\alpha}$ formulas $\psi$ true of $\bar{a}$, $\mathcal{A}$ satisfies the $\Pi_{<\alpha}$ sentences logically equivalent to $(\forall \bar{u})[(\exists\bar{x})\chi(\bar{u},\bar{x})\rightarrow\psi(\bar{u})]$.  Thus, we have a $\Sigma_{<\alpha}$ formula $(\exists\bar{x})\chi(\bar{u},\bar{x})$ that generates the complete $\Pi_{<\alpha}$ type of $\bar{a}$.  
For each $\bar{a}$, let $\varphi_{\bar{a}}(\bar{u})$ be a $\Sigma_{<\alpha}$ formula that generates the complete $\Pi_{<\alpha}$ type of $\bar{a}$.  We claim that these formulas define the orbits.  To show this, it is enough to prove the following lemma.   

\begin{lem}
\label{orbits}

The family $\mathcal{F}$ of finite functions taking $\bar{a}$ to a tuple $\bar{b}$ satisfying $\varphi_{\bar{a}}$ has the back-and-forth property.

\end{lem}

\begin{proof} [Proof of Lemma] 

We first show that for any $\bar{a}$ and $\bar{b}$, if $\bar{b}$ satisfies $\varphi_{\bar{a}}$, then $\bar{a}$ satisfies $\varphi_{\bar{b}}$.  To see this, note that $neg(\varphi_{\bar{b}}(\bar{x}))$ is $\Pi_{<\alpha}$.  If this were true of $\bar{a}$, then it would be true of $\bar{b}$, a contradiction.  Suppose $\bar{b}$ satisfies $\varphi_{\bar{a}}(\bar{x})$.  For any $d$, there exists $c$ such that $\bar{a},c$ satisfies $\varphi_{\bar{b},d}(\bar{x},y)$.  To see this, note that $(\forall y) neg(\varphi_{\bar{b},d}(\bar{x},y))$ is 
$\Pi_{<\alpha}$, so if it were true of $\bar{a}$, then it would also be true of $\bar{b}$, a contradiction.  If $\bar{a},c$ satisfies $\varphi_{\bar{b},d}(\bar{x},y)$, then $\bar{b},d$ satisfies $\varphi_{\bar{a},c}(\bar{x},y)$, so we can go back.  
Now, suppose that $\bar{b}$ satisfies $\varphi_{\bar{a}}(\bar{x})$, and take $c$.  Since $\bar{a}$ satisfies $\varphi_{\bar{b}}(\bar{x})$, the argument above says that there exists $d$ such that $\bar{b},d$ satisfies $\varphi_{\bar{a},c}(\bar{x},y)$, and then $\bar{a},c$ satisfies $\varphi_{\bar{b},d}(\bar{x},y)$.  Therefore, we can go forth.  Hence, for each $\bar{a}$, $\varphi_{\bar{a}}(\bar{x})$ is a $\Sigma_{<\alpha}$ formula that defines the orbit of $\bar{a}$.  This completes the proof of Theorem \ref{MontalbanVariant}.
\end{proof}

Below, we give a pair of examples.  

\bigskip
\noindent
\textbf{Example 1}. Let $\mathcal{A}$ be an ordering of type $\omega^\omega$.  Then the orbits of all tuples in $\mathcal{A}$ are defined by $\Sigma_{<\omega}$ formulas (in fact, the natural defining formulas are computable $\Sigma_{<\omega}$).  However, $\mathcal{A}$ has no $\Pi_\omega$ Scott sentence.

\begin{proof}

We use the following familiar results (see \cite{AK}).

\bigskip
\noindent
\textbf{Facts}:\  

\begin{enumerate}

\item  $\omega^\omega\leq_\omega\omega^{\omega+1}$,  

\item  for each $\beta < \omega^\omega$, there are computable $\Sigma_{<\omega}$ formulas $\lambda(x)$ and $\mu(x,y)$ such that $\lambda(x)$ holds iff the interval to the left of $x$ has order type $\beta$ and $\mu(x,y)$ holds iff the interval between $x$ and $y$ has order type $\beta$.  

\end{enumerate}
It follows from Fact 1 and a well-known result of Karp (see \cite{Keisler} or \cite{KK}) that every $\Pi_\omega$ sentence true of $\omega^\omega$ is true of $\omega^{\omega+1}$.  Therefore, $\omega^\omega$ has no $\Pi_\omega$ Scott sentence.  Take a tuple $\bar{a} = (a_1,\ldots,a_n)$.  Ordinals are rigid, so to define the orbit of $\bar{a}$, we define the tuple itself.  Say that the interval to the left of $a_i$ has type $\beta_i$.  Applying Fact 2, we get a computable $\Sigma_{<\omega}$ formula $\lambda_i(x_i)$ saying that the interval to the left of $x_i$ has type $\beta_i$.  The conjunction of the formulas $\lambda_i(x_i)$ defines the tuple $\bar{a}$.       
\end{proof}

\noindent
\textbf{Example 2}.  Let $\mathcal{A}$ be an expansion of the ordering of type $\omega^\omega$ with a unary predicate $U_0$ for the interval $[0,\omega)$ and unary predicates $U_n$ for the interval $[\omega^n,\omega^{n+1})$, for $n\geq 1$.  Again, we have computable $\Sigma_{<\omega}$ formulas defining the orbits of all tuples.  We have a computable $\Pi_\omega$ Scott sentence.  This is the conjunction of a computable $\Pi_2$ sentence saying $(\forall x)\bigvee_n U_n(x)$, a finitary $\Pi_2$ sentence saying that $<$ is a linear ordering of the universe, with all elements of $U_n$ before all elements of $U_{n+1}$, and computable $\Sigma_{<\omega}$ sentences saying what is the order type of $U_n$.       

\bigskip

We turn to the effective version of Theorem \ref{MontalbanVariant}.  

\begin{thm}
\label{effectiveMontalban}

Let $\mathcal{A}$ be a structure for a computable language $L$ (the structure need not have a computable copy).  Suppose $\alpha\geq 2$.  If $\mathcal{A}$ has a computable $\Pi_{\alpha}$ Scott sentence, then the orbit of each tuple is defined by a computable $\Sigma_{<\alpha}$ formula.    

\end{thm}

\begin{proof}

The proof essentially the same as that for Theorem \ref{MontalbanVariant}.  The corresponding rules for a consistency property involve computable infinitary formulas, so the conjunctions and disjunctions are over c.e. sets of indices.  Our particular consistency property $\mathcal{C}$ will consist of the finite sets of computable $\Pi_\beta$ and computable $\Sigma_\beta$ sentences, for $\beta+1 < \alpha$, such that some interpretation of the constants, mapping distinct constants to distinct elements of $\mathcal{A}$, makes all of the sentences true. These technical changes necessitate no significant change in the argument that constructs, for each tuple $\bar{a}$, a computable 
$\Sigma_{<\alpha}$ formula true of $\bar{a}$ that implies all computable $\Pi_{<\alpha}$ formulas true of $\bar{a}$.    

For each $\bar{a}$, let $\varphi_{\bar{a}}(\bar{x})$ be a $\Sigma_{<\alpha}$ formula that implies all computable $\Pi_{<\alpha}$-formulas true of $\bar{a}$.  To show that for each $\bar{a}$, the orbit is defined by $\varphi_{\bar{a}}$, the following analogue of Lemma \ref{orbits} suffices; the proof is exactly the same as above.

\begin{lem}
\label{orbits.effective}

The family $\mathcal{F}$ of finite functions taking a tuple $\bar{a}$ to a tuple $\bar{b}$ satisfying $\varphi_{\bar{a}}$ has the back-and-forth property.

\end{lem}
\end{proof}

We do not have the effective version for the other implication even in the case where $\alpha$ is a computable successor ordinal.     

\begin{prop}

There is a computable structure $\mathcal{A}$ such that the orbits of all tuples are defined by computable $\Sigma_1$ (even finitary quantifier-free) formulas, but there is no computable $\Pi_2$ Scott sentence. 

\end{prop}

\begin{proof}

The proof owes much to that of Badaev \cite{B}, showing that there is a computable enumeration of a ``discrete'' set of functions that is not ``effectively discrete''.  We start with a computable subtree $T$ of $2^{<\omega}$ with the following features.

\begin{enumerate}

\item  There are no terminal nodes.

\item  There is just one non-isolated path $p$, where this is non-computable.

\end{enumerate}

We may construct $T$ such that for all $\sigma\in T$, $\sigma 0\in T$, and the only non-isolated path has the form $0^{s_0} 1^{k_0}0^{s_1}1^{k_1}\ldots$, where $(k_n)_{n\in\omega}$ is a list of the elements of the halting set, in increasing order, and $s_n$ is the number of steps in the halting computation of $\varphi_{k_n}(k_n)$.
At stage $0$, we put $\emptyset$ into $T$.  Suppose that we have determined $T_s$ at stage $s$, where $T_s$ is the set of nodes in $T$ of length at most $s$.  At stage $s+1$, we add $\sigma 0$ for all $\sigma$ of length $s$.  In addition, we consider $\varphi_{k,s+1}(k)$ for all $k\leq s+1$.  For the computations that halt, we arrange the $k$'s to form $k_0 < k_1 < \ldots < k_r$, and determine the appropriate halting times $s_0,\ldots,s_r$.  We put into $T_{s+1}$ the appropriate initial segment of the sequence $0^{s_0}1^{k_0}\cdots 0^{s_r}1^{k_r}$.  Note that if we are inserting a new $k_i$, it is because $s_i = s+1$, and we already had the appropriate initial segment in $T_s$.  At stage $s+1$, the nodes just described are the only ones that we add to $T_{s+1}$.  Therefore, the tree $T = \bigcup_{s \in \omega} T_s$ is computable.    

We turn the tree $T$ into a class of structures.  The language $\mathcal{L}$ consists of unary predicates $U_n$ for $n\in\omega$. In each $\mathcal{L}$-structure $\mathcal{A}$, we have infinitely many elements $a$ representing each isolated path $f$ in $T$, in that if $f(n) = 0$, then $\mathcal{A}\models \neg{U_na}$ and if $f(n) = 1$, then $\mathcal{A}\models U_na$.  

We can give a computable set of axioms for the elementary first order theory of these structures.  For $\sigma\in T$, we have a finitary quantifier-free formula $\sigma(x)$ that is the conjunction of $U_nx$ for $\sigma(n) = 1$ and $\neg{U_nx}$ for $\sigma(n) = 0$.  For each $n$, we have a finitary quantifier-free formula $T_n(x)$ that is the disjunction over $\sigma\in T\cap 2^n$ of the formulas $\sigma(x)$.  Consider the axioms $(\forall x) T_n(x)$ for all $n$ and $(\exists ^{\geq n} x)\sigma(x)$ for $\sigma\in T$.  Let $T^*$ be the theory generated by these axioms.

The countable models of $T^*$ all have infinitely many elements representing each isolated path.  In addition, they may have one or more elements representing the non-isolated path.  To see that the axioms generate a complete theory, we note that any finitary sentence mentions only finitely many $U_m$, say for $m < n$.  The reducts of the various countable models of $T^*$ to this smaller language are all isomorphic.        

Consider the model of $T^*$ with no elements representing the non-isolated path.  Clearly, this model has a computable copy $\mathcal{A}$ with universe $A = \omega$.  To see this, computably partition the set $A$ into infinitely many infinite, computable sets.  For each $\sigma \in T$, consider the infinite path composed of $\sigma$ followed by all $0$'s; assign all of the elements $a$ from one of the infinite sets in the partition to this path.  Note that if $\sigma_i$ isolates the path represented by $a_i$, then the conjunction of the formulas giving the equalities on the $a_i$ and the formulas $\sigma_i$ generates the complete elementary first order type of $\bar{a}$.  Moreover, since the language has only unary predicates, this formula actually defines the orbit of $\bar{a}$.  Similarly, for any model $\mathcal{D}$ of $T^*$, $\mathcal{A}$ is isomorphic to an elementary substructure of $\mathcal{D}$, and any mapping that sends each element $a \in \mathcal{A}$ to an element $d \in \mathcal{D}$ representing the same isolated path is an elementary embedding.

Let $\mathcal{B}$ be the model of $T^*$ with just one  element $b$ representing the non-isolated path.  We write $\mathcal{A}$ for the substructure of $\mathcal{B}$ isomorphic to the structure $\mathcal{A}$ above.  We want to show that any computable $\Pi_2$ sentence true of $\mathcal{A}$ is true of $\mathcal{B}$.  It is enough to show that any computable $\Sigma_2$ sentence true of $\mathcal{B}$ is true of $\mathcal{A}$.  Take a computable $\Sigma_2$ sentence $\varphi = \bigvee (\exists\bar{u}_i)\varphi_i(\bar{u}_i)$ true of $\mathcal{B}$.  Say $\mathcal{B}\models\varphi_i(b,\bar{a})$, where $\varphi_i(x,\bar{v})$ is computable $\Pi_1$, $\bar{u}_i = x,\bar{v}$.  Let $\delta(\bar{v})$ be a finitary quantifier-free formula generating the type of $\bar{a}$.  Let $p(x,\bar{v})$ be the c.e.\ set of finitary universal conjuncts of $\varphi_i(\bar{u}_i)$.  Let $\Gamma$ consist of the computable set of axioms for $T^*$, plus $\delta(\bar{c})$, plus 
$p(d,\bar{c})$.  

Let $f$ be the non-isolated path through $T$.  We cannot have $\Gamma\vdash\sigma(d)$, for all finite $\sigma\subseteq f$, since then $f$ would be computable.  So, for some $\sigma\subseteq f$, $\Gamma\cup\{\neg{\sigma(d)}\}$ has a model $\mathcal{C}$, where the elements of $\bar{c}$ and $d$ necessarily all represent isolated paths.  Now, by what we noted above, the reduct of $\mathcal{C}$ to the language $\mathcal{L}$ has an elementary substructure isomorphic to $\mathcal{A}$, and we may suppose that $\bar{c}$ and $d$ are in this elementary substructure, and that $\bar{a}$ is mapped to $\bar{c}$.  Therefore, for the element $a'$ mapped to $d$, $\mathcal{A}\models\varphi_i(a',\bar{a})$, as required.    \end{proof} 

The next result gives conditions sufficient to guarantee that a computable structure has a computable $\Pi_{\alpha+1}$ Scott sentence. 

\begin{prop}

Let $\alpha\geq 1$ be a computable ordinal.  Suppose $\mathcal{A}$ is a computable structure, and there is a $\Sigma_\alpha$ Scott family $\Phi$ consisting of computable $\Sigma_\alpha$ formulas, with no parameters.  Then $\mathcal{A}$ has a computable $\Pi_{\alpha+1}$ Scott sentence.    

\end{prop}

\begin{proof}

We can prove the following.  

\noindent
\textbf{Claim}:  There is a computable function taking each tuple $\bar{c}$ to a computable $\Sigma_\alpha$ formula $\varphi_{\bar{c}}(\bar{x})$ that defines the orbit of $\bar{c}$.

\begin{proof} [Proof of Claim]

Let $a$ be a notation for $\alpha$ (see \cite{AK} for a technical exposition of Kleene's system of notations for computable ordinals).  We may suppose that all elements of $\Phi$ have indices of the form $(\Sigma,a,e)$.  Let $R$ be the $\Sigma_\alpha$ relation consisting of pairs $(\bar{c},e)$ such that the formula $\psi_{\bar{c},e}$ with index $(\Sigma,a,e)$ is an element of $\Phi$ that is true of $\bar{c}$.  We can construct a computable sequence of computable $\Sigma_\alpha$ formulas $\tau_{\bar{c},e}$, defined for all $\bar{c}$ and all $e$, built up out of $\top$ and $\bot$, such that $\tau_{\bar{c},e}$ is logically equivalent to $\top$ if $(\bar{c},e)\in R$ and to $\bot$ otherwise.
We let $\varphi_{\bar{c}}(\bar{x})$ be the disjunction over all $e$, of the formulas $\tau_{\bar{c},e}\ \&\ \psi_{\bar{c},e}(\bar{x})$.     
\end{proof}   

Using the formulas $\varphi_{\bar{c}}(\bar{x})$ from the Claim, we can build a computable $\Pi_{\alpha+1}$ Scott sentence as Scott did.  We take the conjunction of the computable $\Sigma_\alpha$ sentence $\varphi_\emptyset$ and the computable $\Pi_{\alpha+1}$ sentences $\rho_{\bar{a}}$ saying  
\[(\forall\bar{x})[\varphi_{\bar{a}}(\bar{x})\rightarrow ((\forall y)\bigvee_b\varphi_{\bar{a},b}(\bar{x},y)\ \&\ \bigwedge_b(\exists y)\varphi_{\bar{a},b}(\bar{x},y))]\ .\]  \end{proof}                            

Recall that in an example above, we showed that for a computable ordering of type $\omega^\omega$, the orbits of all tuples are defined by computable $\Sigma_{<\omega}$ formulas, but there is no $\Pi_\omega$ Scott sentence.   

%%%%%%%%%%%%%%%%%%%%%%%%%%%%%%%%%%%%%%%

\section{Varying the results of A.\ Miller and D.\ Miller} 

A.\ Miller \cite{AMiller} proved that for a countable ordinal $\alpha\geq 2$, if $\mathcal{A}$ has a $\Sigma_\alpha$ Scott sentence and a $\Pi_\alpha$ Scott sentence, then it has one that is $d$-$\Sigma_{<\alpha}$.  We say a little about A.\ Miller's proof.  The case where $\alpha$ is a limit ordinal is trivial; in fact, if $\mathcal{A}$ has a $\Sigma_\alpha$ Scott sentence $\varphi$, then there is a $\Sigma_{<\alpha}$ Scott sentence.  The sentence $\varphi$ is a countable disjunction of formulas $\varphi_i$, each of which is $\Sigma_{<\alpha}$.  One of the disjuncts $\varphi_i$ is true in $\mathcal{A}$, and this is a $\Sigma_{<\alpha}$ Scott sentence.  

The interesting case is as follows.      

\begin{thm} [A.\ Miller]

For a countable ordinal $\alpha\geq 1$, if $\mathcal{A}$ has a Scott sentence that is $\Pi_{\alpha+1}$ and one that is $\Sigma_{\alpha+1}$, then it has one that is $d$-$\Sigma_\alpha$.

\end{thm}

\begin{proof} [Sketch of proof]

A.\ Miller used a result of D.\ Miller \cite{DMiller} saying that for disjoint sets $A,B\subseteq Mod(L)$ both axiomatized by $\Pi_{\alpha+1}$ sentences, there is a separator (i.e., a set containing $A$ and disjoint from $B$) that is a countable union of sets axiomatized by $d$-$\Sigma_\alpha$ sentences.  Suppose that $\mathcal{A}$ has Scott sentences $\varphi$ and $\psi$, where $\varphi$ is $\Pi_{\alpha+1}$ and $\psi$ is $\Sigma_{\alpha+1}$.  Applying the result of D.\ Miller to the disjoint sets $A = Mod(\varphi)$ and $B = Mod(neg(\psi))$, we get a separator which is $Mod(\gamma)$ for some sentence $\gamma$ which is a countable disjunction of $d$-$\Sigma_{\alpha}$ sentences. Since $Mod(neg(\psi))$ is the complement of $Mod(\varphi)$ in $Mod(L)$, $Mod(\gamma) = Mod(\varphi)$; so $\gamma$ is a Scott sentence for $\mathcal{A}$. Thus, $\mathcal{A}$ satisfies one of the disjuncts of $\gamma$, and this is also a Scott sentence for $\mathcal{A}$.
\end{proof}

Our goal is to prove the following.

\begin{thm}
\label{effectiveAMiller}

For a computable ordinal $\alpha\geq 2$, if $\mathcal{A}$ has a Scott sentence that is computable $\Pi_\alpha$ and one that is computable $\Sigma_\alpha$, then there is one that is computable $d$-$\Sigma_{<\alpha}$.   

\end{thm}

Again, the case where $\alpha$ is a limit ordinal is trivial. We want to prove that if $\mathcal{A}$ has one Scott sentence that is computable $\Pi_{\alpha+1}$ and another that is computable $\Sigma_{\alpha+1}$, then there is one that is computable $d$-$\Sigma_\alpha$.

D.\ Miller gave an effective version of his separation theorem, saying that if $A$ and $B$ are disjoint subsets of $Mod(L)$, axiomatized by $\Pi_{\alpha+1}$ sentences in the admissible fragment $L_{\omega_1^{CK}}$, then there is a separator that is a disjoint union of sets axiomatized by $d$-$\Sigma_\alpha$ formulas in $L_{\omega_1^{CK}}$.  This is not good enough for our purposes.  We give a direct proof of the following.

\begin{lem} [Main Lemma]  

Let $\alpha\geq 1$ be a computable ordinal.  If $\mathcal{A}$ has a computable $\Sigma_{\alpha+1}$ Scott sentence $\varphi$ and a computable $\Pi_{\alpha+1}$ Scott sentence $\psi$, then there is a computable $d$-$\Sigma_\alpha$ Scott sentence.

\end{lem}

\begin{proof}

The sentence $\varphi$ has the form $\bigvee_{i\in W} (\exists\bar{u}_i)\varphi_i(\bar{u}_i)$, where each $\varphi_i$ is computable $\Pi_\alpha$, and $W$ is a c.e. set.  For some $i$ and some $\bar{a}$, we have $\mathcal{A}\models\varphi_i(\bar{a})$.  By Theorem \ref{effectiveMontalban}, the orbit of $\bar{a}$ is defined by a computable $\Sigma_\alpha$ formula $\gamma(\bar{u})$.  Note that $(\exists\bar{u})\gamma(\bar{u})$ is computable $\Sigma_\alpha$, and $(\forall\bar{u})(\gamma(\bar{u})\rightarrow\varphi_i(\bar{u}))$ is logically equivalent to a computable $\Pi_\alpha$ sentence.  The conjunction of these is a Scott sentence 
for~$\mathcal{A}$.  
\end{proof}

Below, we give an effective version of D.\ Miller's result, which would suffice to prove Theorem \ref{effectiveAMiller}.    
\begin{thm}
\label{effectiveDMiller}

Let $L$ be a computable language.  For a computable ordinal $\alpha\geq 2$, suppose $A$ and $B$ are disjoint subsets of $Mod(L)$ axiomatized by computable $\Pi_{\alpha}$ sentences.  Then then there is a separator that is the union of a countable family of sets each of which is axiomatized by a computable $d$-$\Sigma_{< \alpha}$ sentence.

\end{thm}
       
\begin{proof}

Let $A = Mod(\varphi)$, and let $B = Mod(\psi)$, where $\varphi$ and $\psi$ are computable $\Pi_\alpha$ sentences.  Say that
$\varphi = \bigwedge_{i \in W} (\forall\bar{u}_i)\varphi_i(\bar{u}_i)$, where $W$ is a c.e. set, and each $\varphi_i(\bar{u}_i)$ has the form 
$\bigvee_{j \in W_{e_i}}(\exists\bar{v}_{i,j}) \xi_{i,j}(\bar{u}_i,\bar{v}_{i,j})$, where each $\xi_{i,j}$ is $\Pi_{\beta}$ for some $\beta$ such that $\beta+1 < \alpha$, and each $W_{e_i}$ is a c.e. set whose index $e_i$ (according to some canonical indexing of the c.e. sets) is determined computably from the index $i$ of $\varphi_i$.  

Let $C$ be an infinite computable set of new Henkin constants, corresponding to the natural numbers.  For each $\mathcal{A}\in A$, let $\mathcal{C}_{\mathcal{A}}$ be the consistency property consisting of the finite sets $S$ of sentences in the language $L\cup C$, each computable $\Sigma_\beta$ or $\Pi_\beta$ for some $\beta$ such that $\beta+1 <\alpha$ (and each, recall, obtained by substituting constants from $C$ for the free variables in a computable $L_{\omega_1\omega}$ formula in normal form), where some interpretation of the constants from $C$ appearing in the sentences of $S$, mapping distinct constants to distinct elements of $\mathcal{A}$, makes all of these sentences true.  Recall the special computable $\Pi_\alpha$ sentence $\varphi$ and its sub-formulas, $\varphi_i$ and $\xi_{i,j}$.  The set $\mathcal{C}_{\mathcal{A}}$ is a consistency property $\mathcal{C}$ satisfying the following condition:   

\begin{itemize}[label=$\star$] 
\item  For each $S\in\mathcal{C}$, for each $i \in W$ and each appropriate $\bar{c}$, there exist $j \in W_{e_i}$ and an appropriate $\bar{d}$ such that for some $S'\supseteq S$ in $\mathcal{C}$, $\xi_{i,j}(\bar{c},\bar{d})\in S'$.
\end{itemize}

For any consistency property $\mathcal{C}$ satisfying $\star$, there is a chain $(S_n)_{n\in\omega}$ of elements of $\mathcal{C}$ such that $\{S_n:n\in\omega\}$ is also a consistency property satisfying $\star$.  For any such chain, the resulting structure is a model of $\varphi$.    

We now consider the other special computable $\Pi_\alpha$ sentence $\psi$.  Say that $\psi = \bigwedge_{i \in W'}(\forall\bar{u}_i)\psi_i(\bar{u}_i)$, where $W'$ is a c.e. set, and for each $i$, $\psi_i(\bar{u}_i) = \bigvee_{j \in W_{e_i}}(\exists\bar{v}_j)\theta_{i,j}(\bar{u}_i,\bar{v}_j)$, where $\theta_{i,j}$ is $\Pi_\beta$ for some $\beta$ such that $\beta+1 < \alpha$, and each $W_{e_i}$ is a c.e. set whose index $e_i$ (according to some canonical indexing of the c.e. sets) is determined computably by the index $i$ of $\varphi_i$.  Since $A = Mod(\varphi)$ and $B = Mod(\psi)$ are disjoint subsets of $Mod(L)$, for any $\mathcal{A}$ satisfying $\varphi$, $\mathcal{C}_{\mathcal{A}}$ cannot satisfy the following added condition, which would witness the truth of $\psi$:

\begin{itemize}[label=$\star \star$] 

\item  For each $S\in\mathcal{C}$, for all $i \in W'$ and all $\bar{c}$ appropriate for $\bar{u}_i$, there exists $S'\supseteq S$ such that for some $j \in W_{e_i}$ and $\bar{d}$, we have $\theta_{i,j}(\bar{c},\bar{d})\in S'$.  

\end{itemize}

It follows that there must exist $S\in\mathcal{C}_{\mathcal{A}}$, and some $i$ and $\bar{c}$ appropriate for $\bar{u}_i$, such that for all $j$ in the c.e.\ set $W_{e_i}$ and all $\bar{d}$, $S\cup\{\theta_{i,j}(\bar{c},\bar{d})\}$ is not satisfied by any assignment in $\mathcal{A}$.  Let $\bar{c}'$ be the tuple of constants from $C$, other than $\bar{c}$, that appear in $S$, and let $\chi(\bar{c},\bar{c}')$ be the conjunction of $S$ and sentences expressing that the tuples $\bar{c}, \bar{c}'$ are disjoint and the elements of $\bar{c}'$ are distinct. 

Now, although $\bar{c}$ is appropriate for the tuple of variables $\bar{u}_i$, it might be that the tuple $\bar{c}$ assigns the same constant to multiple variables; i.e, the tuple $\bar{c}$ could list the same constant multiple times. Therefore, to define $\chi(\bar{u}_i, \bar{x})$ unambiguously, for a given constant in $\bar{c}$, substitute the variable from $\bar{u}_i$ with least index to which this element of $\bar{c}$ was assigned.  Finally, let $\rho(\bar{u}_i, \bar{x})$ be the conjunction of $\chi(\bar{u}_i, \bar{x})$ and formulas that express the pairwise equality of any elements of $\bar{u}_i$ to which the same constant in $\bar{c}$ was assigned, and the pairwise inequality of any elements of $\bar{u}_i$ to which different constants in $\bar{c}$ were assigned.  Then 
\[\mathcal{A}\models (\forall\bar{u}_i)[((\exists\bar{x})\rho(\bar{u}_i,\bar{x})) \rightarrow (\bigwedge_{j \in W_{e_i}}(\forall\bar{v}_{i,j})(neg(\theta_{i,j}(\bar{u}_i,\bar{v}_{i,j}))))]\] .

\noindent Note that $(\exists\bar{u}_i \bar{x})\rho(\bar{u}_i,\bar{x})$ is computable $\Sigma_{<\alpha}$, and \[(\forall\bar{u}_i)[((\exists\bar{x})\rho(\bar{u}_i,\bar{x})) \rightarrow (\bigwedge_{j \in W_{e_i}}(\forall\bar{v}_{i,j})(neg(\theta_{i,j}(\bar{u}_i,\bar{v}_{i,j}))))]\] is logically equivalent to a computable $\Pi_{<\alpha}$ sentence.  Both sentences are true in $\mathcal{A}$.  They cannot both be true in any model of $\psi$, for then there would be a tuple satisfying $neg(\psi_i(\bar{u}_i))$.  The conjunction gives a computable $d$-$\Sigma_{<\alpha}$ sentence that is true in $\mathcal{A}$ and not true in any model of $\psi$.  Let $M_{\mathcal{A}}$ be the class of models for this sentence. As our separator, we take the union of the sets $M_\mathcal{A}$.  While there may be uncountably many models $\mathcal{A}$ of $\varphi$, there are only countably many pairs of computable infinitary sentences. Hence, our separator is the union of a countable family of sets $S_{\mathcal{A}}$.  
\end{proof} 

\section{Finitely generated groups}       

Knight and Saraph \cite{KS} observed that every computable finitely generated group has a computable $\Sigma_3$ Scott sentence.  However, for many kinds of computable finitely generated groups, there is a computable $d$-$\Sigma_2$ Scott sentence.  In particular, this is so for finitely generated free groups \cite{Free}, finitely generated Abelian groups, the infinite dihedral group of rank $2$  \cite{KS}, further variants of the dihedral group \cite{Raz}, polycyclic groups, and certain groups of interest in geometric group theory (lamplighter and Baumslag-Solitar groups) \cite{Ho}.  
Based on the known examples, Ho and Knight had conjectured that every computable finitely generated group has a computable $d$-$\Sigma_2$ Scott sentence.  Knight also conjectured that every finitely generated group (not necessarily computable) has a $d$-$\Sigma_2$ Scott sentence (not necessarily computable $d$-$\Sigma_2$).  
Recently, Harrison-Trainor and Ho \cite{HH} gave an example of a computable finitely generated group that does not have a $d$-$\Sigma_2$ Scott sentence, thereby disproving both conjectures.     

Here we give necessary and sufficient conditions for a finitely generated group to have a $d$-$\Sigma_2$ Scott sentence.  We also give necessary and sufficient conditions for a computable finitely generated group to have a computable $d$-$\Sigma_2$ Scott sentence.  We show that for a finitely generated group, there is a $d$-$\Sigma_2$ Scott sentence iff for some generating tuple, the orbit is defined by a $\Pi_1$ formula iff for each generating tuple, the orbit is defined by a $\Pi_1$ formula.  For a computable finitely generated group, there is a computable $d$-$\Sigma_2$ Scott sentence iff for some generating tuple, the orbit is defined by a computable $\Pi_1$ formula iff for each generating tuple, the orbit is defined by a computable $\Pi_1$ formula.      
\subsection{Finitely generated groups with a $d$-$\Sigma_2$ Scott sentence}

In \cite{KS}, it is observed that a computable finitely generated group has a computable $\Sigma_3$ Scott sentence.  
Throughout the rest of this section, we use the notation $\langle\bar{x}\rangle \cong \langle\bar{y}\rangle$ to represent the computable $\Pi_1$ formula that says $\bar{x}$ and $\bar{y}$ satisfy the exact same relators and non-relators.  Note that if $\bar{a}$ is a fixed tuple of a group $G$, then the set of relators and non-relators satisfied by $\bar{a}$ is a set computable in $G$.  The notation $\langle\bar{x}\rangle \cong \langle\bar{a}\rangle$ represents the conjunction of the formulas of the forms $\bar{w}(\bar{x}) = 1$ and $\bar{w}(\bar{x})\not= 1$ that are true of $\bar{a}$ in $\mathcal{A}$.  This formula should not be thought of as including $\bar{a}$ as parameters; instead, it represents a $\Pi_1$ formula in~$\bar{x}$.

\begin{prop}

Every finitely generated group has a $\Sigma_3$ Scott sentence.  

\end{prop}

\begin{proof}

Let $G$ be a group with generating tuple $\bar{a}$.  As in \cite{KS}, we get a Scott sentence saying that $(\exists \bar{x})[\langle\bar{x}\rangle\cong \langle\bar{a}\rangle\ \&\ (\forall y)\bigvee_w w(\bar{x}) = y]$.
\end{proof}

We can prove the following.       

\begin{thm}
\label{characterization}

For a finitely generated group $G$, the following are equivalent:

\begin{enumerate}

\item  $G$ has a $\Pi_3$ Scott sentence,

\item  $G$ has a $d$-$\Sigma_2$ Scott sentence

\item  for some generating tuple, the orbit is defined by a $\Pi_1$ formula

\item  for each generating tuple, the orbit is defined by a $\Pi_1$ formula. 

\end{enumerate} 

\end{thm}

\begin{proof}

Clearly, $(2)\Rightarrow (1)$.  Using the result of A.\ Miller, together with the fact that $G$ has a $\Sigma_3$ Scott sentence, we get $(1)\Rightarrow (2)$.  To complete the proof, we will show that $(1)\Rightarrow (4)\Rightarrow (3)\Rightarrow (1)$.  For $(1)\Rightarrow (4)$, suppose $G$ has a $\Pi_3$ Scott sentence, and let $\bar{a}$ be a generating tuple.  By the result of Montalb\'{a}n, the orbit of $\bar{a}$ is defined by a computable $\Sigma_2$ formula 
$\varphi(\bar{x}) = \bigvee_i(\exists\bar{u}_i)\varphi_i(\bar{x},\bar{u}_i)$, where 
$\varphi_i$ is $\Pi_1$.  Take $i$ and $\bar{b}$ such that $G\models\varphi_i(\bar{a},\bar{b})$.  For some tuple of words $\bar{w}$, $G\models\bar{w}(\bar{a}) = \bar{b}$.  Then the orbit of $\bar{a}$ is defined by the $\Pi_1$ formula 
$\varphi_i(\bar{x},\bar{w}(\bar{x}))$.  Clearly, $(4)\Rightarrow (3)$.  To show that $(3)\Rightarrow (1)$, suppose $\bar{a}$ is a generating tuple with orbit defined by a $\Pi_1$ formula $\psi(\bar{u})$.  We show that for all tuples $\bar{b}$, the orbit is defined by a $\Sigma_2$ formula.  Suppose 
$G\models\bar{b} = \bar{w}(\bar{a})$.  Then the orbit of $\bar{b}$ is defined by the 
$\Sigma_2$-formula
$\varphi(\bar{x}) = (\exists\bar{u})(\psi(\bar{u})\ \&\ \bar{x} = \bar{w}(\bar{u}))$.  Then by the result of Montalb\'{a}n, $G$ has a $\Pi_3$ Scott sentence.      
\end{proof}

In this proof of Theorem \ref{characterization} above, we used Miller's result to show that if $G$ has a $\Pi_3$ Scott sentence, then it has a $d$-$\Sigma_2$ Scott sentence.  There is an alternative proof, using the following result of Ho \cite{Ho}.  

\begin{lem} [Generating Set Lemma]

Let $G$ be a computable finitely generated group, and suppose $\varphi(\bar{x})$ is a computable $\Sigma_2$ formula, satisfied in $G$, such that all tuples satisfying $\varphi(\bar{x})$ generate $G$.  Then $G$ has a computable $d$-$\Sigma_2$ Scott sentence.  

\end{lem} 

\begin{proof}

Let $G = \langle \bar{a} \rangle$, where $\bar{a}$ satisfies the formula $\varphi(\bar{x})$.  

Ho's Scott sentence is the conjunction of the following:
\begin{enumerate}

\item  the computable $\Pi_2$ sentence saying $(\forall\bar{x})[\varphi(\bar{x})\rightarrow (\forall y)\bigvee_{w} w(\bar{x}) = y]$,

\item the computable $\Sigma_2$ sentence saying $(\exists \bar{x})[\varphi(\bar{x})\ \&\ \langle\bar{x}\rangle\cong \langle \bar{a} \rangle]$.
\end{enumerate}
\end{proof} 

We automatically have the following non-effective analogue of Ho's result.

\begin{lem}

Suppose $G$ is a finitely generated group, and there is a $\Sigma_2$ formula $\varphi(\bar{x})$, satisfied in $G$, and such that all tuples satisfying $\varphi(\bar{x})$ generate $G$.  Then $G$ has a $d$-$\Sigma_2$ Scott sentence.

\end{lem}  

\begin{proof} [Alternative proof of Theorem \ref{characterization}]

If there is a $\Pi_3$ Scott sentence, then by the result of Montalb\'{a}n, there is a $\Sigma_2$ formula defining the orbit of a generating tuple.  Then by the analogue of the result of Ho, there is a $d$-$\Sigma_2$ Scott sentence.  
\end{proof} 

\subsection{Computable finitely generated groups with a computable $d$-$\Sigma_2$ Scott sentence} 

In \cite{KS}, it is observed that every computable group has a computable $\Sigma_3$ Scott sentence.  We can prove the following.   

\begin{thm}
\label{characterization.effective}

For a computable finitely generated group $G$, the following are equivalent:

\begin{enumerate}

\item  there is a computable $\Pi_3$ Scott sentence,

\item  there is a computable $d$-$\Sigma_2$ Scott sentence,

\item  for some generating tuple, the orbit is defined by a computable $\Pi_1$ formula,

\item  for each generating tuple, the orbit is defined by a computable $\Pi_1$ formula.  

\end{enumerate}

\end{thm}

\begin{proof}

We show that $(2)\Rightarrow (1)\Rightarrow (4)\Rightarrow (3)\Rightarrow (2)$.
For $(2)\Rightarrow (1)$, we note that a computable $d$-$\Sigma_2$ Scott sentence may be regarded as a computable $\Pi_3$ Scott sentence.  
For $(1)\Rightarrow (4)$, if $G$ has a computable $\Pi_3$ Scott sentence, and 
$\bar{a}$ is a generating tuple, then by Theorem \ref{effectiveMontalban}, the orbit of $\bar{a}$ is defined by a computable $\Sigma_2$ formula $\varphi(\bar{x}) = \bigvee_{i \in W}(\exists\bar{u}_i)\varphi_i(\bar{x},\bar{u}_i)$, where each $\varphi_i$ is computable $\Pi_1$, and $W$ is some c.e. set.  For some $i \in W$ and some $\bar{b}$, $G\models\varphi_i(\bar{a},\bar{b})$.  For some tuple of words $\bar{w}$, $G\models\bar{w}(\bar{a}) = \bar{b}$.
Then the orbit of $\bar{a}$ is defined by the computable $\Pi_1$ formula 
$\varphi_i(\bar{x},\bar{w}(\bar{x}))$.
Clearly, $(4)\Rightarrow (3)$.
For $(3)\Rightarrow (2)$, let $\bar{a}$ be a generating tuple whose orbit is defined by a computable $\Pi_1$ formula.  We may regard this as a computable $\Sigma_2$ formula.  By the result of Ho, $G$ has a computable $d$-$\Sigma_2$ Scott sentence.  
\end{proof}

\subsection{Example of Harrison-Trainor and Ho}

Ho and Harrison-Trainor \cite{HH} gave the definition below. They considered not just finitely generated groups, but more general finitely generated structures.  We consider only groups.

\begin{defn} [Harrison-Trainor-Ho]

A finitely generated group $G$ is \emph{self-reflective} if there is a generating tuple $\bar{a}$ and a tuple $\bar{b}$ generating a proper subgroup $H$, such that
\[(G,\bar{a})\cong (H,\bar{b})\] 
and every existential formula true of $\bar{b}$ in $G$ is true of $\bar{b}$ in $H$.  

\end{defn} 

\begin{prop}

Let $G$ be a finitely generated group.  Then $G$ is self-reflective iff there is a generating tuple $\bar{a}$ whose orbit is not defined by a $\Pi_1$ formula.

\end{prop} 

\begin{proof}

$\Rightarrow$  Suppose that $G$ is self-reflective, witnessed by $\bar{a}$ and $\bar{b}$, where $\bar{a}$ generates $G$, $\bar{b}$ generates a proper subgroup $H$, and all existential formulas true of $\bar{b}$ in $G$ are also true of $\bar{b}$ in $H$.  Since $(G,\bar{a})\cong (H,\bar{b})$, these existential formulas are also true of $\bar{a}$ in $G$. Therefore, the universal formulas true of $\bar{a}$ in $G$ are also true of $\bar{b}$ in $G$.  This means that any $\Pi_1$ formula true of $\bar{a}$ is also true of $\bar{b}$. Then the orbit of $\bar{a}$ is not defined by a $\Pi_1$ formula.

$\Leftarrow$  Suppose that $G$ is generated by a tuple $\bar{a}$ whose orbit is not defined by a $\Pi_1$ formula.  Let $\varphi(\bar{x})$ be the $\Pi_1$ formula obtained as the conjunction of the universal formulas true of $\bar{a}$.  There must be some $\bar{b}$ satisfying $\varphi(\bar{x})$ and not generating $G$.  Let $H$ be the subgroup generated by $\bar{b}$.  Then $(G,\bar{a})\cong (H,\bar{b})$, and all existential formulas true of $\bar{b}$ in $G$ are true of $\bar{a}$ in $G$ and hence true of $\bar{b}$ in $H$.  This means that $G$ is self-reflective.     
\end{proof} 

Harrison-Trainor and Ho constructed a computable finitely generated group $G$ that is self-reflective.  They also showed that for any such group, the index set is $m$-complete $\Sigma^0_3$.  In this way, they arrived at the fact that their group $G$ has no computable $d$-$\Sigma_2$ Scott sentence.  Relativizing, they got the fact that for any set $X$, the set of $X$-computable indices for copies of $G$ is $m$-complete $\Sigma^0_3$ relative to $X$, so there is no $X$-computable $d$-$\Sigma_2$ Scott sentence.  It follows that the group $G$ has no $d$-$\Sigma_2$ Scott sentence.    
\section{Problems}

\begin{enumerate}

\item  Is there a finitely presented group that is self-reflective?  

\item  Is there a precise sense in which most finitely generated groups are not self-reflective?  Can we say in terms of limiting density that the typical finitely generated group has a $d$-$\Sigma_2$ Scott sentence?   

\item  Is there a computable finitely generated group with a $d$-$\Sigma_2$ Scott sentence but no computable $d$-$\Sigma_2$ Scott sentence?  

\item  Give necessary and sufficient conditions for a computable structure $\mathcal{A}$ to have a computable 
$\Pi_{\alpha+1}$ Scott sentence.  

\end{enumerate}  
     
%%%%%%%%%%%%%%%%


\begin{thebibliography}{99} 

\normalsize
\baselineskip=17pt

\bibitem{AK}  C.\ J.\ Ash and J.\ F.\ Knight, \emph{Computable Structures and the Hyperarithmetical Hierarchy}, Elsevier, 2000.  

\bibitem{B}  S.\ Badaev, ``Computable enumerations of families of general recursive functions'', \emph{Algebra and Logic}, vol.\ 16(1977), pp.\ 83-98.  

\bibitem{Free}  J.\ Carson, V.\ Harizanov, J.\ Knight, K.\ Lange, C.\ Maher, C.\ McCoy, A.\ Morozov, S.\ Quinn, and J.\ Wallbaum, ``Describing free groups'', \emph{TAMS}, vol.\ 364(2012), pp.\ 5715-5728.

%\bibitem{G}  S.\ S.\ Goncharov, ``The quantity of non-autoequivalent constructivizations'', \emph{Algebra and Logic}, vol.\ 16(1977), pp.\ 257-282.  

\bibitem{HH}  M.\ Harrison-Trainor and M-C\ Ho, ``Finitely generated groups'', pre-print.  

\bibitem{Ho}  M-C Ho, ``Describing groups'', \emph{PAMS}, to appear in \emph{PAMS}. 

\bibitem{Keisler}  H.\ J.\ Keisler, \emph{Model Theory for Infinitary Logic}, North-Holland, 1971.

\bibitem{KK}  J.\ F.\ Knight and H.\ J.\ Keisler, ``Barwise: infinitary logic and admissible sets'', \emph{Bull.\ Symbolic Logic}, vol.\ 10(2004), pp.\ 4-36.

\bibitem{KS}  J.\ F.\ Knight and V.\ Saraph, ``Scott sentences for certain groups'', to appear in \emph{Archive for Math.\ Logic}.  

%\bibitem{MC}  C.\ McCoy, \emph{Relativization, Categoricity, %and Dimension}, PhD thesis, University of Notre Dame, 2000.  

\bibitem{AMiller}  A.\ Miller, ``The Borel classification of the isomorphism class of a countable model'',  \emph{NDJFL}, vol.\ 24(1983), pp.\ 22-34. 

\bibitem{DMiller}  D.\ E.\ Miller, ``The invariant $\mathbf{\Pi^0_\alpha}$ separation principle'', \emph{TAMS}, vol.\ 242(1978), pp.\ 185-204.

\bibitem{Montalban} A.\ Montalb\'{a}n, ``A robuster Scott rank'', \emph{PAMS}, vol.\ 143(2015), pp.\ 5427-5436.

\bibitem{Raz}  A.\ Raz, ``Index sets of some computable groups'', honors thesis, Wellesley College, 2014,  http://repository.wellesley.edu/thesiscollection/223/ 

\bibitem{Scott}  D.\ Scott, ``Logic with denumerably long formulas and finite strings of quantifiers'', in \emph{The Theory of Models}, ed.\ by J.\ Addition, L.\ Henkin, and A.\ Tarski, 1965, North-Holland, Amsterdam, pp.\ 329-341.

%\bibitem{S}  V. L.\ Selivanov, ``Enumerations of families of general recursive functions'', \emph{Algebra and Logic}, vol.\ 15(1976), pp.\  127-141.

\bibitem{VB}  M.\ Vanden Boom, ``The effective Borel hierarchy'', \emph{Fund.\ Math.}, vol.\ 195(2007), pp.\ 269-289. %new

\bibitem{Vaught}  R.\ L.\ Vaught, ``Invariant sets in topology and logic'', \emph{Fund.\ Math.}, vol.\ 82(1974), pp.\ 269-294. 

\end{thebibliography}
\end{document}